\documentclass[11pt]{amsart}
\usepackage{amssymb}
\usepackage{amsfonts}
\usepackage{amsmath}
\usepackage[utf8]{inputenc}
\usepackage[mathscr]{eucal}
\usepackage{xypic}
\usepackage{graphicx}
\usepackage[english]{babel}
\usepackage{epsfig}
\usepackage{mathrsfs}
\usepackage[all]{xy}
\usepackage{tikz}
\usetikzlibrary{matrix, arrows, plotmarks}
\usepackage{pgfplots}
\def\g{{\rm{g}}}

\newtheorem{theorem}{Theorem}

\newtheorem{corollary}[theorem]{Corollary}

\newtheorem{lemma}[theorem]{Lemma}
\newtheorem{proposition}[theorem]{Proposition}
\theoremstyle{definition}

\newtheorem{example}[theorem]{Example}

\newtheorem{algorithm}[theorem]{Algorithm}

\def\max{{\rm max}}
\def\min{{\rm min}}
\def\msg{{\rm msg}}

\def\m{{\rm{m}}}

\def\e{{\rm{e}}}
\def\e{{\rm{e}}}
\def\g{{\rm{g}}}
\def\n{{\rm{n}}}
\def\F{{\rm{F}}}

\def\N{{\mathbb N}}
\def\Z{{\mathbb Z}}

\def\F{{\rm F}}

\begin{document}

\title[Numerical semigroups with concentration two]{Numerical semigroups with concentration two} 

\author{José C. Rosales}
\address{Departamento de \'Algebra, Universidad de Granada, E-18071 Granada, Spain}
\email{jrosales@ugr.es}

\author{Manuel B. Branco}
\address{Departamento de Matemática, Universidade de Évora, 7000 Évora, Portugal}
\email{mbb@uevora.pt}

\author{Márcio A. Traesel}
\address{Departamento de Matemática, Instituto Federal de São Paulo, Caraguatatuba, SP, Brazil}
\email{marciotraesel@ifsp.edu.br}

\thanks{The first author was partially supported by MTM-2017-84890-P and by Junta de Andalucia group FQM-343. The second author is supported by the project FCT PTDC/MAT/73544/2006). 2010 Mathematics Subject Classification: 20M14, 11D07.}

\begin{abstract}
We define the concentration of a numerical semigroup $S$ as $\mathsf{C}(S)=\max \left\{\text{next}_S(s)-s ~|~ s\in S \backslash \{0\}\right\}$ wherein  $\text{next}_S(s)=\min\left\{x \in S ~|~ s<x\right\}$.
In this paper, we study  the class of numerical semigroups with concentration $2$. We give algorithms to calculate the whole set of this class of semigroups with given multiplicity, genus or Frobenius number.  Separately, we prove that this class of  semigroups verifies the Wilf's conjecture.
\end{abstract}


\maketitle

\section{Introduction}\label{S1}

Let $\Z$ be the set of integers an let $\N=\left\{n\in \Z ~|~ n\geq 0\right\}$ the set of nonnegative integers. A submonoid of  $(\N, +)$  is a subset of $\N$ closed addition and containing $0$.  A numerical semigroup is  a submonoid $S$ of $(\N, +)$  such that $\N\backslash S=\left\{n\in\N ~|~ n\not\in S\right\}$  is  finite.

If  $S$ numerical semigroup and $s$ an element in  $S$,  we denote by $\text{next}_S(s)=\min\left\{x \in S ~|~ s<x\right\}$.   We define the concentration of a numerical semigroup $S$ as $\mathsf{C}(S)=\max \left\{\text{next}_S(s)-s ~|~ s\in S \backslash \{0\}\right\}$. The least nonnegative integer belonging to $S$ is called the multiplicity, denoted by $\m(S)$. 
Clearly, we have that if $S$ is a numerical semigroup with concentration $1$ then $S=\left\{0,\m(S), \rightarrow \right\}$. 
 If $m$ is a positive integer, then the semigroup $\left\{0,\m,\rightarrow \right\}$ is denoted here by $\triangle(m)$ and it is called half-line or ordinary.

Our aim in this paper is the study the numerical semigroups with concentration $2$.  

If $\mathcal  X$  is a nonempty subset of   $\N$, we denote by $\left\langle \mathcal X\right\rangle$ the submonoid of $(\N,+)$ generated by $\mathcal X$, that is, 
\[\left\langle \mathcal X\right\rangle=\left\{\sum_{i=1}^n \lambda_i\,x_i ~|~ n\in\N\backslash\left\{0\right\}, ~ x_1,\ldots , x_n\in \mathcal X,\text{and} ~ \lambda_1,\ldots, \lambda_n\in \N  \right\},\]

which is a numerical semigroup if and only if $\gcd(\mathcal X)=1$  (see \cite{libro}).

If $M$ is a submonoid  of $(\N,+)$ and $M=\left\langle \mathcal X\right\rangle$ then we say that $\mathcal X$ is a system of generators of $M$. Moreover, if $M\neq \left\langle \mathcal Y\right\rangle$ for all $\mathcal Y \varsubsetneq \mathcal X$, then  we say that $\mathcal X$ is a minimal system of generators of $S$. 
In  \cite[Corollary 2.8]{libro} it is shown that every submonoid of $(\N,+)$ has a unique minimal system of generators, which is finite. We denote by $\msg (M)$ the minimal systemm of generators of $M$, its cardinality is called the embedding dimension of $M$ and it is denoted by $\e(M)$.

This paper is organized as follows. In Section $2$ we give a characterization of numerical semigroups with concentration $2$ in terms of its minimal system of generators.
If $m\in \N\backslash\{0,1\}$ we denote by $\mathsf{C}_2[m]$ the set of all numerical semigroups with concentration $2$ and  multiplicity $m$,  that is, 
\[\mathsf{C}_2[m]= \{ S ~|~S \text{ is a numerical semigroup}, \mathsf{C}(S)=2 ~\text{and}~ \m(S)=m\}.\]

In this section we will order the elements $\mathsf{C}_2[m]$ making a rooted tree. This ordering will provide  us an algorithmic procedure that allows us to recurrently build the elements $\mathsf{C}_2[m]$.

Let $S$ be a numerical semigroup. As  $\N\backslash S$ is finite,  there exist integers  $\F(S)=\max \left\{z\in \Z ~|~ z\not\in S\right\}$ and   the cardinality  of $\N\backslash S$ denoted by $\g(S)$,  which are  two  important invariants of $S$ called Frobenius number and genus of $S$, respectively. See  for instance \cite{alfonsin} and \cite{barucci} to understand the importance of the study of these invariants.

We started section $3$ by seeing that $\mathsf{C}_2[m]$ is a finite set if and only if $m$ is odd. Besides we give an algorithm that allows us  compute all elements of $\mathsf{C}_2[m]$ with a given genus.

Given $S$ a numerical semigroup, we denote by $N(S)=\left\{s\in S ~|~ s<\F(S)\right\}$ and its cardinality is denoted by $\n(S)$.

In $1978$, Wilf conjectured (see \cite{wilf}) that if $S$ is a numerical semigroup  then $\g(S)\leq (\e(S)-1) \n(S)$. This question is still widely open and it is  one of the most important problems in numerical semigroups theory. A very good source of the state of the art of this problem  is  \cite{delgado}.  Our aim in  section $4$ will be to prove that  numerical semigroups  with concentration $2$  verify the Wilf's conjecture.

By using the  terminology of \cite{pacific}, a numerical semigroup is irreducible if it cannot be expressed as the intersection of two numerical semigroups properly containing it. A numerical semigroup is a symmetric numerical semigroup (pseudo-symmetric, resp.) if is irreducible and its Frobenius number is odd (even, resp). 
This class of numerical semigroups are probably the numerical semigroups that have been more studied in the literature (see \cite {kunz} and \cite {barucci}). 

Given a positive integer $F$, denote by 
\[\mathsf{C}_2(F)= \{ S ~|~S \text{ is a numerical semigroup}, \mathsf{C}(S)=2 ~\text{and} ~ \F(S)=F\}\]
\[\text{and} ~  I\big(\mathsf{C}_2(F)\big)= \{ S\in \mathsf{C}_2(F)~|~ S ~ \text{is a irreducible numerical semigroup}\}.\]

In section $5$ we define an equivalence relation $\sim$ over $\mathsf{C}_2(F)$ such that
 $\mathsf{C}_2(F)/\! \sim=\left\{[S]~| ~ S\in I\big(\mathsf{C}_2(F)\big)\right\}$ where [S] denotes the equivalence class of S
with respect to $\sim$. 
Hence, to compute all the elements in  $\mathsf{C}_2(F)$ it is enough to determine all  elements in $I\big(\mathsf{C}_2(F)\big)$ and, for each $S\in I\big(\mathsf{C}_2(F)\big)$, to compute the class $[S]$. As a consequence of this study we give an algorithm that allows us to calculate the whole set of $\mathsf{C}_2(F)$.


\section{The tree associated to  $\mathsf{C}_2[m]$}\label{S2}

We started this section by presenting several characterizations for the  numerical semigroups with concentration $2$.

\begin{proposition}\label{1}
Let $S$ be a numerical semigroup such that  $S$ is not half-line. The following conditions are equivalent:
\begin{enumerate}
\item $\mathsf{C}(S)=2$.
\item  $h+1\in S$  for all $h\in \N\backslash S$ such that $h>\m(S)$.
\item  $\left\{s+1, s+2\right\}\cap S\neq \emptyset$ for all $s\in S\backslash \{0\}$.
\item  $\left\{x+1, x+2\right\}\cap S\neq \emptyset$ for all $x\in\msg (S)$.
\end{enumerate}
\end{proposition}
\begin{proof}
 $1) ~ implies ~ 2).$ Let $s\in S$ such that $s< h <  \text{next}_S(s)$. Since $s\neq 0$, we have that that $h> \m(S)$ and thus  $\text{next}_S(s)-s\leq 2$. Hence $h+1= \text{next}_S(s)\in S$.

 $2) ~ implies ~ 3).$ If $s+1\in\N\backslash S$  and $s+1> \m(S)$, then by $2)$, we conclude that $s+2\in S$.

$3)~ implies ~4).$ Trivial.

$4)~ implies ~1).$ Suppose that $\msg (S)=\left\{n_1,n_2,\ldots ,n_e\right\}$. If $s\in S\backslash \{0\}$, then there exists 
$(\lambda_1,\ldots, \lambda_e)\in \N^{e}\backslash \{(0,\ldots,0)\}$ such that  $s=\lambda_1 n_1+\cdots + \lambda_e n_e$. Let $\lambda_i\neq 0$ with $i\in \left\{1,\ldots, e\right\}$. As by hypothesis $\left\{n_i+1, n_i+2\right\}\cap S\neq \emptyset$,  if $n_i+1\in S$ then $s+1=\lambda_1 n_1+\cdots +(\lambda_i-1)n_i+\cdots + \lambda_e n_e+n_i+1$ and thus $s+1\in S$. In the same way, if $n_i+2\in S$ we obtain that $s+2\in S$. Hence $\text{next}_S(s)-s\leq 2$, that is, $\mathsf{C}(S)=2$.

\end{proof}

\begin{example}\label{2}
Using the previous proposition we deduce that $S=\left\langle 5,7,9\right\rangle$ is a numerical semigroup with $\mathsf{C}(S)=2$, because $\left\{5+2, 7+2, 9+1\right\}\subseteq S$. 
\end{example}

Given  $m$ belongs to $\N \backslash\{0,1\}$, we denote by \\
 $\mathsf{C}_2[m]=\{ S ~|~S \text{ is a numerical semigroup}, \mathsf{C}(S)=2 ~\text{and}~ \m(S)=m\}$ and \\
 $\overline {\mathsf{C}_2[m]}= \{ S ~|~S \text{ is a numerical semigroup},  \mathsf{C}(S)\leq 2 ~\text{and}~ \m(S)=m\}$,

The next result  characterize the set $\overline{\mathsf{C}_2[m]}$ and it has an immediate prove. 

\begin{proposition}\label{3}
If $m\in \N\backslash \{0,1\}$, then $\overline {\mathsf{C}_2[m]}= \mathsf{C}_2[m] \cup \left\{\triangle (m)\right\}$.
\end{proposition}

From this result it is easy to prove.

\begin{lemma}\label{4}
If $m\in \N\backslash \{0,1\}$ and $S\in \mathsf{C}_2[m]$, then  $S\cup \{F(S)\}\in \overline {\mathsf{C}_2[m]}$.
\end{lemma}

The previous result enable us, given an element $S \in  \overline {\mathsf{C}_2[m]}$,  to define  recursively the following sequence of elements in $\overline {\mathsf{C}_2[m]}$:

\begin{itemize}
\item $S_0=S$,
\item $S_{n+1}=\left\{ \begin{array}{ll}
  S_n\cup\left\{\F(S_n)\right\} & \hbox{if } S_n \neq \triangle(m)  \\
  \triangle(m) & \hbox{otherwise}.
    \end{array}\right.$
\end{itemize}

 The next result  can be easily proved.

\begin{proposition}\label{5}
If $m\in \N\backslash \{0,1\}$, $S\in \overline {\mathsf{C}_2[m]}$ and $\left\{S_n ~|~n\in\N\right\}$ is the previous sequence of numerical semigroups, then there exists $k\in\N$ such that $S_k=\triangle(m)$.
\end{proposition}

A graph $G=(V ,E)$ consists of a set denoted by $V$ and a collection $E$ of ordered pairs $(v,w)$ of distinct elements from $V$. Each element of $V$ is called a vertex and each element of $E$ is called an edge.  A path of length $n$ connecting the vertices $u$ and $v$ of $G$ is a sequence of distinct edges of the form $(v_0,v_1), (v_1,v_2),\ldots, (v_{s-1},v_s)$ with $v_0 = u$ and $v_s = v$. 

A graph $G$ is a tree if there exists a vertex $r$ (known as the root of $G$) such that for every other vertex $v$ of $G$, there exists a path connecting $v$ and $r$. If $(u,v)$ is a edge of the tree then we say that $u$ is a son of $v$.

We define the graph $G(\overline {\mathsf{C}_2[m]})$ as graph whose vertices are the elements of $\overline {\mathsf{C}_2[m]}$ and $(S, T)\in \overline {\mathsf{C}_2[m]}\times \overline {\mathsf{C}_2[m]}$ is an edge if $T=
S\cup \left\{\F(S)\right\}$.  As a consequence of Proposition \ref{5}, we deduce the following.

\begin{theorem}\label{6}
If $\m\in \N\backslash \{0,1\}$, then the graph $G(\overline {\mathsf{C}_2[m]})$ is a tree rooted in $\triangle(m)$.  
\end{theorem}

The previous results allows us to construct recursively the elements of the set $\overline {\mathsf{C}_2[m]}$. From the root $\triangle(m)$ in each step we are connecting each of the vertex with its sons. We will characterize the sons of an arbitrary vertex of this tree, for that we need the following result.

\begin{lemma} \cite[Lemma 1.7]{colloquium}\label{7}
Let $S$ be a numerical semigroup and $x\in S$. Then $S\backslash\{x\}$ is a numerical semigroup if and only if $x\in \text{msg}(S)$.
\end{lemma}

\begin{proposition}\label{8}
Let $m\in \N\backslash \{0,1\}$ and $S\in \overline {\mathsf{C}_2[m]}$.  Then the set of sons of $S$ in the  tree $G(
 \overline {\mathsf{C}_2[m]}$ is equal to  $\left\{S\backslash \{x\} ~|~ x\in \text{msg}(S), ~ x\geq \F(S)+2 \right\}$ ,
\end{proposition}

\begin{proof}
If $x\in \text{msg}(S)$ and $x\geq \F(S)+2$, then by applying Proposition \ref{1} and Lemma \ref{7} we have that  $S\backslash \{x\} \in {\mathsf{C}_2[m]}$. Hence $S\backslash \{x\}$ is a son of $S$ with  $F(S\backslash \{x\})=x$.

Conversely, if $T$ is a son of $S$, then $T\in \overline {\mathsf{C}_2[m]}$ and $S=T\cup \{F(T)\}$. Hence we deduce that $T=S\backslash \{F(T)\})$. By Lemma \ref{7},  we have that $F(T)\in \text{msg}(S)$ and  $F(S)<F(T)$. Since $T\in \overline {\mathsf{C}_2[m]}$ then, by Proposition \ref{1}, we obtain that $F(T)-1\in T$. Therefore, $F(T)-1\in S $ and consequently $F(T)\geq F(S)+2$.
\end{proof}

\begin{example}\label{9}
Let us construct  the tree  $G(\overline {\mathsf{C}_2[3]})$. 

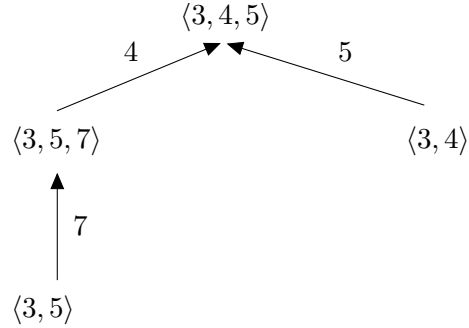
\vspace{0.5cm}
\begin{tikzpicture}[line cap=round,line join=round,>=triangle 45,xscale=0.75,yscale=0.75]
\clip(0,-4) rectangle (15,2);

\draw [<-]  (8.9,1.3) -- (6,0.1);
\draw [<-]  (9,1.29) -- (12.5,0.2);
\draw [<-]  (6,-1) -- (6,-2.9);

\draw (8,2.2) node[anchor=north west] {$\langle 3,4,5\rangle$};

\draw (5,0) node[anchor=north west] {$\langle 3,5,7\rangle$};
\draw (12,0) node[anchor=north west] {$\langle 3,4\rangle$};

\draw (5,-3) node[anchor=north west] {$\langle 3,5\rangle$};

\draw (10.8,1.5) node[anchor=north west] {$5$};
\draw (7,1.5) node[anchor=north west] {$4$};
\draw (6.1,-1.6) node[anchor=north west] {$7$};
\end{tikzpicture}
\end{example}

The number that appears on either side of the edges is the element that we remove from the semigroup to obtain its corresponding son.  Note that, this number coincide with the Frobenius number  of the new son.


\section{The genus of the elements in  $\mathsf{C}_2[m]$}\label{S3}

It is clear that, in  the tree  $G(\overline {\mathsf{C}_2[m]})$,  the elements of $\mathsf{C}_2[m]$ with minimum  genus are the sons of $\triangle (m)$. Consequently, we obtain the following result.

\begin{proposition}\label{10}
If $m\in \N\backslash \{0,1\}$ and $S\in \mathsf{C}_2[m]$, then  $\g(S)\geq m$. Furthermore, we have
 the following equality  $\left\{S\in \mathsf{C}_2[m] ~|~ \g(S)=m\right\}=\left\{\triangle (m)\backslash \{m+i\} ~|~i\in \{1,\ldots, m-1\right\}$.
\end{proposition}

From the previous characterization it is natural to ask which are the elements of $\mathsf{C}_2[m]$ with maximum genus.
As a consequence of the next proposition, we will see that if $m$ is even then $\mathsf{C}_2[m]$ contains elements of any genus greater than or equal to $m$.

If $S$ is a numerical semigroup, then  $\N\backslash S$ is a finite set and thus it is to deduce  our next result.

\begin{lemma} \label{11}
If $S$ is a numerical semigroup,  then the set $\left\{T ~|~ T ~ \text{is a numerical semigroup and} ~ S\subseteq T\right\}$ is  finite. 
\end{lemma}

\begin{proposition}\label{12}
Let $m\in \N\backslash \{0,1\}$. Then $\mathsf{C}_2[m]$ is finite if and only if $m$ is odd.
\end{proposition}
\begin{proof}
$Necessity$. Given $m$ is even and $n\in \N$ denote by $S(n)=\langle \{m\}+\left\{2.k ~|~ k\in\N\right\}\rfloor\rangle\cup \left\{n,\rightarrow\right\}$. Clearly, we have that $S(n)$ is an element of $\mathsf{C}_2[m]$ for all $n\geq m+2$ and so $\mathsf{C}_2[m]$ is an infinite set. 

$Sufficiency$. If $S\in \mathsf{C}_2[m]$, then by Proposition \ref{1}, we deduce that $\left\{m+1, m+2\right\}\cap S\neq\emptyset$. Hence, either $\langle m, m+1\rangle\subseteq S$  or $\langle m, m+2\rangle\subseteq S$. Since $m$ is odd we have that $\langle m, m+1\rangle$ and $\langle m, m+2\rangle$ are numerical semigroups. Therefore, we can conclude that  $\mathsf{C}_2[m]\subseteq \left\{T ~|~ T \text{ is a numerical semigroup and}~\langle m, m+1\rangle\subseteq T\right\}\cup$ $\left\{T ~|~ T \text{ is a numerical semigroup and}~\langle m, m+2\rangle\subseteq T\right\}$. By applying now Lemma \ref{11} we get that $\mathsf{C}_2[m]$ is a finite set.
\end{proof}

As a consequence of the previous proposition, we obtain that following result.
\begin{corollary}\label{13}
If $m\in \N\backslash \{0,1\}$ such that $m$ is even, then the  set of the genus of the elements in  $\mathsf{C}_2[m])$
 is equal to $\left\{m,\rightarrow\right\}$.
\end{corollary}

Now our aim is to give an algorithm to compute all elements in the set $\mathsf{C}_2[m]$ with fixed genus.  To this end, we need to introduce some concepts and results.

Let $G$ be a rooted tree and $v$ one of its vertices. We define the depth of the vertex $v$ as the length of the path that connects $v$ to the root of $G$, denoted by $d(v)$. If $k\in\N$, we denote by 
 \[N(G,k)=\left\{v ~|~ d(v)=k\right\}.\]

We define the height of the tree $G$ by $h(G)=\ max\left\{k\in\N ~|~ N(G,k)\neq\emptyset \right\}$.

The next result is easy to prove.
\begin{proposition}\label{14}
 Let $m\in \N\backslash \{0,1\}$ and $k\in\N$. Then the following conditions hold.
\begin{enumerate}
\item $N\big(G(\overline {\mathsf{C}_2[m]}\big),k)=\left\{S\in  \overline {\mathsf{C}_2[m]} ~|~ \g(S)=m-1+k\right\}$. 
\item $N\big(G(\overline {\mathsf{C}_2[m]}),k+1\big)=\left\{S ~|~  S ~\text{is a son of an element in} ~N\big(G(\overline {\mathsf{C}_2[m]}),k)\big) \right\}$.
\item  If $m$ is odd, then \\ $\left\{\g(S) ~| ~ S\in  \mathsf{C}_2[m]\right\}=\left\{m, m+1\ldots m+ h\big(G(\overline {\mathsf{C}_2[m]})\big)-1 \right\}$. 
\end{enumerate}
\end{proposition}

We are already in conditions to present the announced algorithm jointly with an example.

\begin{algorithm}\label{15}\mbox{}\par
\textsc{Input}: Integers $m, g$  such that $1 \leq m-1\leq g$.\par
\textsc{Output}:  The set $\left\{S\in \overline {\mathsf{C}_2[m]}  ~|~ \g(S)=g\right\}$
\begin{enumerate}
\item $A=\left\{\langle m, m+1, \ldots 2m-1\rangle\right\}$, $i=m-1$.
\item If $i=g$ then return $A$. 
\item For each $S\in A$  compute
$B_S=\left\{T ~|~ T ~\text{is a son of} ~ S\in G(\overline {\mathsf{C}_2[m]}) \right\}$.
\item If  $~ \bigcup_{S\in A}B_S=\emptyset$, then return $\emptyset$.
\item $A:= \bigcup_{S\in A} B_S$, $i=i+1$ and go to step $2$.
\end{enumerate}
\end{algorithm}

\begin{example}\label{16}
Let us compute the set $\left\{S\in \overline {\mathsf{C}_2[4]} ~|~ \g(S)=5\right\}$.

\begin{enumerate}
\item Start with $A=\langle 4,5,6,7\rangle$, $i=3$.
\item  the first loop constructs $B_{\langle 4,5,6,7 \rangle}=\left\{\langle 4,6,7,9 \rangle, \langle 4,5,7 \rangle, \langle 4,5,6 \rangle\right\}$ 
then  $A=\left\{\langle 4,6,7,9 \rangle, \langle 4,5,7 \rangle, \langle 4,5,6 \rangle\right\}$, $i=4$.
\item the second loop constructs $B_{\langle 4,6,7,9\rangle}=\left\{\langle 4,6,9,11\rangle, \langle 4,6,7\rangle\right\}$,  $B_{\langle 4,5,7\rangle}=\emptyset$ and $B_{\langle 4,5,6\rangle}=\emptyset$
then  $A=\left\{\langle 4,6,9,11\rangle, \langle 4,6,7\rangle\right\}$, $i=5$.
\end{enumerate}

Hence  $\left\{S\in \overline {\mathsf{C}_2[4]} ~|~ \g(S)=5\right\}=\left\{\langle 4,6,9,11\rangle, \langle 4,6,7\rangle\right\}$.
\end{example}

We finished this section  by putting two problems :
\begin{enumerate}
\item What is the cardinality of $\mathsf{C}_2[m]$ if  $m$ is odd belongs to  $ \N\backslash \{0,1\}$?
\item What is the height of the tree $G(\overline {\mathsf{C}_2[m]})$ if  $m$ is odd belongs to  $ \N\backslash \{0,1\}$?
\end{enumerate}

\section{Wilf's conjecture}\label{S4}

Wilf's conjecture is one of combinatorial problems related to numerical semigroups  and despite substantial progress remains open in the general case. Our first aim in this section is to prove that every numerical semigroup with concentration $2$  satisfies Wilf's conjecture.

Using the terminology introduced in \cite{elementales} a numerical semigroup $S$ is elementary if $F(S)<2\m(S)$.
Let us start by recall the following result of Kaplan  in  \cite[Proposition $26$]{kaplan}.

\begin{lemma}\label{17} 
Every elementary numerical semigroup satisfies Wilf's conjecture.
\end{lemma}

As a consequence of \cite{Sylvester} and \cite{froberg} we have the following result.

\begin{lemma}\label{18} 
If $S$ is a numerical semigroup with $\e(S)\in \{2,3\}$, then $S$ satisfies Wilf's conjecture.
\end{lemma}

For any finite  set $X$, $\#X$ denotes the cardinal of $X$.

\begin{lemma}\label{19} 
If $S\in \mathsf{C}_2[m]$  and $\F(S)>2m$, then $\n(S)\geq \frac{m}{2}+2$. 
\end{lemma}
\begin{proof}
Let $A=\left\{m=a_1<a_2<\cdots <2m=a_k\right\}=\left\{s ~|~s\in S ~\text{and}~ m\leq s\leq 2m\right\}$. Since $A\subseteq N(S)\backslash\{0\}$ we get that $\n(S)\geq \#A+1$. On the other hand,  as $S\in \mathsf{C}_2[m]$ then $a_{i+1}-a_i\leq 2$ for all $i\in\{1,\ldots, k-1\}$. Then we have that $m=(a_k-a_{k-1})+(a_{k-1}-a_{k-2})+\cdots +(a_2-a_1)\leq 2(k-1)$. Therefore $\#A = k\geq \frac{m}{2}+1$  and thus $\n(S)\geq \frac{m}{2}+2$.
\end{proof}

\begin{theorem}\label{20} 
 Every numerical semigroup with concentration $2$ satisfies Wilf's conjecture.
\end{theorem}
\begin{proof}
 Taking into account Lemmas \ref{17} and \ref{18}, we assume that $\F(S)>2m$ and $\e(S)\geq 4$. We need to show that if $S\in\mathsf{C}_2[m]$ then $\g(S)\leq (\e(S)-1) \n(S)$. By Proposition \ref{1} we have that, if $h\in \N\backslash S$ and $h\geq m$ then $h+1\in S$.
 Therefore, the correspondence
 \[f: \{h\in \N\backslash S~|~h\geq m\} \rightarrow N(S)\backslash \{0\},\]
defined by $f(h)=h+1$ if $h\neq \F(S)$ and $f\big(\F(S)\big)=m$ is an  injective  map. Hence $\g(S)\leq m+\n(S)-2$. As by Lemma \ref{19}  $\n(S)\geq \frac{m}{2}+2$  this forces $2\n(S)\geq m+4\geq m-2$. Then we obtain that $\g(S)\leq m+\n(S)-2 \leq 3\n(S)\leq (\e(S)-1) \n(S)$, because $\e(S)\geq 4$.
This proves that $S$ verifies Wilf's Conjecture.
\end{proof}

Taking advantage of the introduction of elementary numerical semigroups, in this section, we give an algorithm to compute the set all
 elementary numerical semigroups with concentration $2$ and  multiplicity $m$, that is,
\[\mathsf{EC}_2[m]= \{ S ~|~S\in\mathsf{C}_2[m]  ~\text{and} ~S ~\text{ is an elementary numerical semigroup}\}.\]

The next result is easy to prove and it can be deducted of [\cite{zhao}, Proposition  $2.1$].

\begin{lemma}\label{21}
  Let $m\in \N\backslash \{0,1\}$ and  let $A\subseteq \left\{m+1,\ldots,2m-1\right\}$. Then $\left\{0,m\right\}\cup A\cup \left\{2m,\rightarrow\right\}$ is an elementary numerical semigroup with multiplicity $m$. Furthermore, every elementary numerical semigroup   with multiplicity $m$ is of this form.
\end{lemma}

Given  $m\in\N \backslash\{0,1\}$, we denote by 
 \[\overline {\mathsf{EC}_2[m]}= \{ S ~|~S \text{ is elementary semigroup},  ~ \mathsf{C}(S)\leq 2 ~\text{and}~ \m(S)=m\}.\]

It is easy to prove our next result.

\begin{lemma}\label{22}
Let $m\in \N\backslash \{0,1\}$. Then the following conditions hold:
\begin{enumerate}
\item  $\overline {\mathsf{EC}_2[m]}= \mathsf{EC}_2[m] \cup \left\{\triangle (m)\right\}$.
\item $S\in \mathsf{EC}_2[m]$, then  $S\cup \{F(S)\}\in \overline {\mathsf{EC}_2[m]}$.
\end{enumerate}
\end{lemma}

Given $S\in \overline {\mathsf{EC}_2[m]}$,  by using Lemma \ref{22}, we can define recursively the following sequence of elements in $\overline {\mathsf{EC}_2[m]}$.

\begin{itemize}
\item $S_0=S$,
\item $S_{n+1}=\left\{ \begin{array}{ll}
  S_n\cup\left\{\F(S_n)\right\} & \hbox{if } S_n \neq \triangle(m)  \\
  \triangle(m) & \hbox{otherwise}.
    \end{array}\right.$
\end{itemize}

The next result has an immediate prove.

\begin{lemma}\label{23}
If $m\in \N\backslash \{0,1\}$, $S\in \overline {\mathsf{EC}_2[m]}$ and $\left\{S_n ~|~n\in\N\right\}$ is the previous sequence of numerical semigroups, then there exists $k\in\N$ such that $S_k=\triangle(m)$.
\end{lemma}

We can define a new graph $G(\overline {\mathsf{EC}_2[m]})$ as graph whose vertices are the elements of $\overline {\mathsf{EC}_2[m]}$ and $(S, T)\in \overline {\mathsf{EC}_2[m]}\times \overline {\mathsf{EC}_2[m]}$ is an edge if $T=
S\cup \left\{\F(S)\right\}$.  

As a consequence of Lemma \ref{23} and Proposition \ref{8} we have the following result.

\begin{proposition}\label{24}
If $m\in \N\backslash \{0,1\}$, then the graph $G(\overline {\mathsf{EC}_2[m]})$ is a tree rooted in $\triangle(m)$.  Moreover, the set of sons of the vertice $S$ in the tree is the set  $\left\{S\backslash \{x\} ~|~ x\in \text{msg}(S), ~ \F(T)+2\leq x\leq  2m-1\right\}$.
\end{proposition}

\begin{example}\label{25}
Let us construct  the tree  $G(\overline {\mathsf{EC}_2[4]})$. 

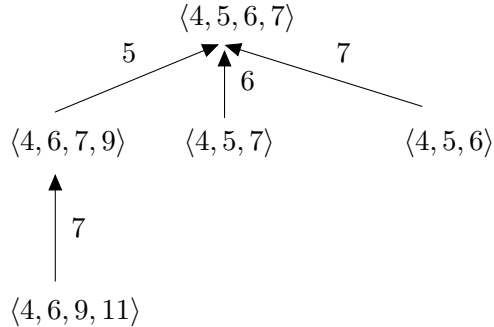
\vspace{0.5cm}
\begin{tikzpicture}[line cap=round,line join=round,>=triangle 45,xscale=0.75,yscale=0.75]
\clip(0,-5) rectangle (15,2);

\draw [<-]  (8.9,1.3) -- (6,0.1);
\draw [<-]  (9,1.29) -- (12.5,0.2);
\draw [<-]  (9,1.2) -- (9,0);
\draw [<-]  (6,-1) -- (6,-2.9);

\draw (8,2.2) node[anchor=north west] {$\langle 4,5,6,7\rangle$};

\draw (5,0) node[anchor=north west] {$\langle 4,6,7,9\rangle$};
\draw (8.1,0) node[anchor=north west] {$\langle 4,5,7\rangle$};
\draw (12,0) node[anchor=north west] {$\langle 4,5,6\rangle$};

\draw (5,-3) node[anchor=north west] {$\langle 4,6,9,11\rangle$};

\draw (10.8,1.5) node[anchor=north west] {$7$};
\draw (7,1.5) node[anchor=north west] {$5$};
\draw (9.1,1) node[anchor=north west] {$6$};
\draw (6.1,-1.6) node[anchor=north west] {$7$};
\end{tikzpicture}
\end{example}

On the same line as the previous section,  we finished this section  by putting two problems :
\begin{enumerate}
\item What is the cardinality of $\mathsf{EC}_2[m]$ if $m$ belongs to $\N\backslash \{0,1\}$?
\item What is the height of the tree $G(\overline {\mathsf{EC}_2[m]})$ if $m$ belongs to $\N\backslash \{0,1\}$?
\end{enumerate}

\section{The Frobenius number}\label{S5}

Our aim in this section is to give an algorithm  to compute  the whole set of numerical semigroups  with concentration $2$  and with fixed  Frobenius number.

\begin{proposition}\label{26}\cite[Lemma $4$  ]{computer}\label{30}
Let $S$ be a numerical semigroup with Frobenius number $F$. Then:
\begin{enumerate}
\item $S$ is irreducible if and only if $S$ is maximal  in the set of all the numerical semigroups with Frobenius number $F$.
\item If $h=\max\left\{x\in \N\backslash S ~|~ F-x\not\in S~\text{and}~x\neq \frac{F}{2}\right\}$, then $S\cup\{h\}$ is a numerical semigroups with Frobenius number $F$.
\item $S$ is irreducible if and only if $\left\{x\in \N\backslash S ~|~ F-x\not\in S~\text{and}~x\neq \frac{F}{2}\right\}=\emptyset$.
\end{enumerate}
\end{proposition}

The following result has immediate prove.
\begin{lemma}\label{27}
Let $S$ be a numerical semigroup with concentration $2$, $x\in \N\backslash S$, $x\neq \F(S)$ and  $S\cup\{x\}$ is a numerical semigroup, then $S\cup\{x\}$ is a numerical semigroup with concentration $2$ and  Frobenius number $\F(S)$.
\end{lemma}

Given  $F\in \N \backslash\{0,1\}$, we denote by 
 \[\mathsf{C}_2(F)=\{ S ~|~S \text{ is a numerical semigroup}, \mathsf{C}(S)=2 ~\text{and}~ \F(S)=F\}.\]

Let $S$ be  non-irreducible numerical semigroup. Denote by  
\[\alpha (S)=\max\left\{x\in \N\backslash S ~|~ \F(S)-x\not\in S~\text{and}~x\neq \frac{\F(S)}{2}\right\}.\]
As a consequence of Lemma \ref{27} and $2)$ of Proposition \ref{26}, we can define recurrently the following sequence of elements of $\mathsf{C}_2(F)$:

\begin{itemize}
\item $S_0=S$,
\item $S_{n+1}=\left\{ \begin{array}{ll}
  S_n\cup\{\alpha(S_n)\} & \hbox{if } S_n ~ \hbox{is non-irreducible } \\
  S_n & \hbox{otherwise}.
    \end{array}\right.$
\end{itemize}
 
Taking into account the previous results  the next result it easy to prove.

\begin{proposition}\label{28}
Let $F\in \N \backslash\{0,1\}$, $S\in \mathsf{C}_2(F)$  and let $\left\{ S_n ~|~n\in\N\right\}$ be the previous sequence. Then there exists a positive integer $k$ such that $S_k$ is an irreducible numerical semigroup.
\end{proposition}

We will call $S_k$ the  irreducible numerical semigroup associated to $S$ and it will be denoted by $\mathscr V(S)$.

We define the following equivalence relation over $\mathsf{C}_2(F)$:
\[S \sim  T  ~~ \text{if and only if} ~~ \mathscr V(S)  =\mathscr V(T).\] 
 We denote the equivalence class of $S\in \mathsf{C}_2(F)$ modulo $\sim$ by $[S]=\left\{T\in \mathsf{C}_2(F) ~|~ S \sim  T \right\}$ and the quotient set  
 $\mathsf{C}_2(F)/\! \sim\,=\left\{[S] ~|~ S\in \mathsf{C}_2(F) \right\}$. 

 Denote by $\mathcal I\big(\mathsf{C}_2(F)\big)=\left\{S\in \mathsf{C}_2(F) ~|~ S ~\text{is irreducible}\right\}$.

As a consequence of Proposition \ref{28} we have the following result.

\begin{theorem}\label{29}
  If $F\in \N \backslash\{0,1\}$,  then  the quotient set  $\mathsf{C}_2(F)/\! \sim\,=\left\{[S] ~|~ S\in \mathcal I\big(\mathsf{C}_2(F)\big)\right\}$. Moreover, if $\{S, T\}\subseteq \mathcal I\big(\mathsf{C}_2(F)\big)$ and  $S\neq T$  then $[S]\cap[T]=\emptyset$.
\end{theorem}

In view of Theorem \ref{29}, in order to determine explicitly the elements in the set $\mathsf{C}_2(F)$ we need:
\begin{itemize}
\item[1)] an algorithm to compute the set $\mathcal I\big(\mathsf{C}_2(F)\big)$;
\item[2)] an algorithm to compute the class $[S]$, for each  $S\in \mathcal I\big(\mathsf{C}_2(F)\big)$.
\end{itemize}

In \cite{forum} it is given an efficient procedure to compute the set of irreducible numerical semigroups with Frobenius number $F$. Using Proposition \ref{1}, we obtain that a numerical semigroup is or is not of  concentration $2$. Therefore we have solved 1).

Now we will focus on solving 2).
Let $\bigtriangleup\in \mathcal I\big(\mathsf{C}_2(F)\big)$. We define the graph $G([\bigtriangleup])$ whose vertices are the elements of $[\bigtriangleup]$ and $(S,T) \in [\bigtriangleup] \times [\bigtriangleup]$ is an edge if and only if  $T=S\cup \{\alpha(S)\}$.

By definition, when $S$ is irreducible we say that $\alpha (S)=+\infty$, because in this case  $\alpha (S)$ does not exist.

\begin{proposition}\label{30} 
 If $F\in \N \backslash\{0,1\}$  and $\bigtriangleup\in \mathcal I\big(\mathsf{C}_2(F)\big)$, then $G([\bigtriangleup])$  is a tree rooted in $\bigtriangleup$. 
 Moreover,  the set of sons of vertex $T$ is equal to  

$\left\{T\backslash \{x\}~|~x\in \msg(T), ~ \frac{F}{2}<x<F, ~ \alpha (T)<x ~\text{and}~ \right.\\
  {} \hspace{6cm} \left. \{x-1,x+1\}\subseteq T  ~\text{or}~ x=\m(T)\right\}$.
\end{proposition}
\begin{proof}
If $S$ is a son $T$, then $T=S\cup \{\alpha(S)\}$ and thus $S=T\backslash\{\alpha (S)\}$. By Lemma \ref{7}, we have that $\alpha (S)\in \msg(T)$. It is clear that  $\frac{F}{2}<\alpha(S)<F$ and $\alpha(S)=\m(T)$ or $\left\{\alpha(S)-1,\alpha(S)+1\right\}\subseteq T$. Also we have that $\alpha(T)<\alpha(S)$.

Conversely, if $x\in\msg(T)$, $\frac{F}{2}<x<F$ and  $\left\{x-1,x+1\right\}\subseteq T$ or $x=\m(T)$ then $T\backslash\{x\}\in \mathsf{C}_2(F)$.  If $\alpha(T)<x$ then $\alpha(T\backslash \{x\})=x$. Hence $T=\left(T\backslash \{x\}\right)\cup \left(\alpha(T\backslash \{x\}\right)$ and so $T\backslash \{x\}$ is a son of $T$.
\end{proof}

\begin{example}\label{31}
 Applying Proposition \ref{26}, we have that  $\bigtriangleup=\langle 5,6,7,8\rangle\in \mathcal I(\mathsf{C}_2(9))\big)$. Now by applying Proposition \ref{30}, let us construct $G([\bigtriangleup])$.

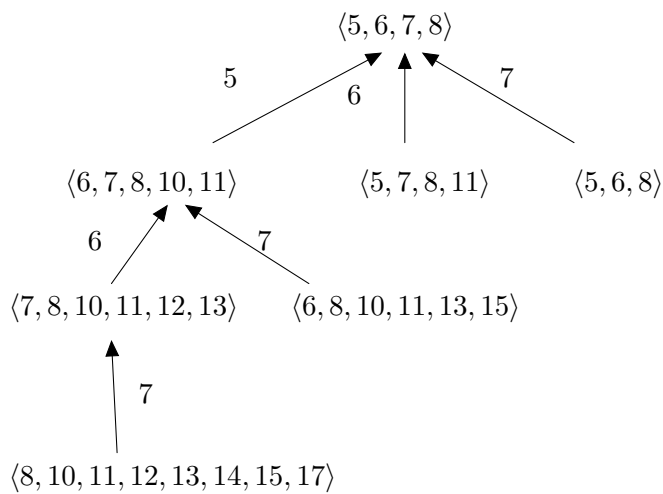
\begin{tikzpicture}[line cap=round,line join=round,>=triangle 45,xscale=0.75,yscale=0.75]
\clip(0,-5) rectangle (17,5);

\draw [<-] (8.8,4.1) -- (5.8,2.5);
\draw [<-] (9.2,4.1) -- (9.2,2.5);
\draw [<-] (9.5,4.1) -- (12.2,2.5);

\draw [<-] (5.3,1.4) -- (7.5,0);
\draw [<-] (5,1.4) -- (4,0);

\draw [<-] (4,-1) -- (4.1,-3);

\draw (7.8,5) node[anchor=north west] {$\langle 5, 6, 7, 8\rangle$};
\draw (3,2.2) node[anchor=north west] {$\langle 6,7,8,10,11\rangle$};
\draw (8.2,2.2) node[anchor=north west] {$\langle 5,7,8,11\rangle$};
\draw (12,2.2) node[anchor=north west] {$\langle 5,6,8\rangle$};
\draw (2,0) node[anchor=north west] {$\langle 7,8,10,11,12,13\rangle$};
\draw (7,0) node[anchor=north west] {$\langle 6,8,10,11,13,15\rangle$};
\draw (2,-3) node[anchor=north west] {$\langle 8,10,11,12,13,14,15,17\rangle$};

\draw (5.8,4) node[anchor=north west] {$5$};
\draw (8,3.7) node[anchor=north west] {$6$};
\draw (10.7,4) node[anchor=north west] {$7$};

\draw (3.4,1.1) node[anchor=north west] {$6$};
\draw (6.4,1.1) node[anchor=north west] {$7$};

\draw (4.3,-1.6) node[anchor=north west] {$7$};
\end{tikzpicture}
\end{example}
 The numbers that appears on either side of the edges is the elements that we remove from the semigroup to obtain its son.


\end{document}